\documentclass[12pt,a4paper,reqno]{amsart}
\usepackage[margin=2.8cm,footskip=1cm]{geometry}
\usepackage[utf8]{inputenc}
\usepackage[T1]{fontenc}
\usepackage{comment}
\usepackage{bbm}
\usepackage[english]{babel}
\usepackage{amsmath}
\usepackage{amsfonts}
\usepackage{amssymb}
\usepackage{amsthm}
\usepackage{ucs}
\usepackage{graphicx}
\graphicspath{ {./figures/} }
\DeclareGraphicsExtensions{.pdf}
\usepackage{xcolor}
\usepackage{dsfont}
\usepackage{tikz}
\usepackage{wasysym}
\usepackage{physics}
\usepackage{mathrsfs} 
\usepackage{mathtools}
\usepackage{xifthen}
\usepackage[textwidth=2.5cm,textsize=tiny]{todonotes}
\usepackage{enumitem}
\usepackage{stmaryrd}
\usepackage{constants}
\usepackage[ruled]{algorithm2e}
\SetKwInput{KwInput}{Input}
\SetKwInput{KwOutput}{Output}
\usepackage{tikz,pgfplots}
\usepackage[hidelinks]{hyperref}
\usepackage[font={small}]{caption}
\usepackage{pbox}
\newcommand{\g}{\mathfrak{g}}

\usepackage{constants}

\newconstantfamily{ctrlcst}{
	symbol=\kappa
}
\newconstantfamily{csts}{
	symbol=C
}
\renewconstantfamily{normal}{
	symbol=c
}

\makeatletter
\newcommand{\pushright}[1]{\ifmeasuring@#1\else\omit\hfill\(\displaystyle#1\)\fi\ignorespaces}
\newcommand{\pushleft}[1]{\ifmeasuring@#1\else\omit\(\displaystyle#1\)\hfill\fi\ignorespaces}
\makeatother

\renewcommand{\norm}[1]{\|#1\|}
\newcommand{\normI}[1]{\left\|#1\right\|_{\scriptscriptstyle 1}}

\renewcommand{\emptyset}{\varnothing}
\newcommand{\setof}[2]{\{#1\,:\,#2\}}
\newcommand{\bsetof}[2]{\bigl\{#1\,:\,#2\bigr\}}
\newcommand{\Bsetof}[2]{\Bigl\{#1\,:\,#2\Bigr\}}

\newcommand{\Z}{\mathbb{Z}}
\newcommand{\R}{\mathbb{R}}
\newcommand{\Rd}{\mathbb{R}^d}
\newcommand{\Zd}{\mathbb{Z}^d}

\newcommand{\p}{\mathbb{P}}

\newcommand{\sfe}{\mathsf{e}}
\newcommand{\sfo}{\mathsf{o}}
\newcommand{\sfO}{\mathsf{O}}

\newcommand{\bbE}{\mathbb{E}}

\newcommand{\bbG}{\mathbb{G}}

\newcommand{\bbJ}{\mathbb{J}}

\newcommand{\bbP}{\mathbb{P}}

\newcommand{\bbR}{\mathbb{R}}
\newcommand{\bbS}{\mathbb{S}}

\newcommand{\bbZ}{\mathbb{Z}}

\newcommand{\calC}{\mathcal{C}}

\newcommand{\calF}{\mathcal{F}}

\newcommand{\calW}{\mathcal{W}}

\newcommand{\Cov}{\mathrm{Cov}}

\renewcommand{\nleftrightarrow}{\mathrel{\ooalign{\(\leftrightarrow\)\cr\hidewidth\(/\)\hidewidth}}}

\newcommand{\Wulff}{\mathscr{W}}

\newcommand{\betac}{\beta_{\mathrm c}}
\newcommand{\betasat}{\beta_{\rm sat}}

\newcommand{\f}{{\scriptscriptstyle\rm f}}

\newcommand{\betahat}{\hat{\beta}_{\rm sat}}

\newcommand{\n}{\mathbf{n}}

\newcommand{\trn}{\langle \sigma_{0};\sigma_{x}\rangle}

\theoremstyle{plain}
\newtheorem{theorem}{Theorem}[section]
\newtheorem{lemma}[theorem]{Lemma}

\newtheorem{corollary}[theorem]{Corollary}
\newtheorem{conjecture}[theorem]{Conjecture}
\newtheorem{remark}{Remark}[section]

\newtheorem{obs}{Observation}

\newtheorem{openproblem}[theorem]{Open problem}

\author{Yacine Aoun}
\address{Section de Mathématiques, Université de Genève, Rue du Conseil-Général 7-9, 1205 Genève, Switzerland}
\email{Yacine.Aoun@unige.ch}

\author{Kamil Khettabi}
\address{Section de Mathématiques, Université de Genève, Rue du Conseil-Général 7-9, 1205 Genève, Switzerland}
\email{kamil.khettabi@gmail.com}

\date{\today}

\title[Two-point function of the Ising model with infinite-range interactions]{On the two-point function of the Ising model with infinite range-interactions}
\begin{document}
\begin{abstract}
In this article, we prove some results concerning the truncated two-point function of the infinite-range Ising model above and below the critical temperature. More precisely, if the coupling constants are of the form $J_{x}=\psi(x)\sfe^{-\rho(x)}$ with $\rho$ some norm and $\psi$ an subexponential correction, we show under appropriate assumptions that given $s\in\bbS^{d-1}$, the Laplace transform of the two-point function in the direction $s$ is infinite for $\beta=\betasat(s)$ (where $\betasat(s)$ is a the biggest value such that the inverse correlation length $\nu_{\beta}(s)$ associated to the truncated two-point function is equal to $\rho(s)$ on $[0,\betasat(s)))$. Moreover, we prove that the two-point function satisfies  Ornstein-Zernike asymptotics for $\beta=\betasat(s)$ on $\bbZ$. As far as we know, this constitutes the first result on the behaviour of the two-point function at $\betasat(s)$. Finally, we show that there exists $\beta_{0}$ such that for every $\beta>\beta_{0}$, $\nu_{\beta}(s)=\rho(s)$. 
All the results are new and their proofs are built on different results and ideas developed in~\cite{Duminil-Copin+Tassion-2016, Aoun+Ioffe+Ott+Velenik-CMP-2021}.
\end{abstract}

\maketitle	

\section{Introduction}

In the present paper, we study the behaviour of the two-point function in Ising models with infinite-range interactions. In~\cite{Aoun+Ioffe+Ott+Velenik-CMP-2021} (see also~\cite{Aoun+Ioffe+Ott+Velenik-PRE-2021}), the first author and collaborators considered a general class of lattice spin systems (including the Ising model) on $\mathbb{Z}^{d}$ with interactions of the form $J_{x}=\psi(x)\sfe^{-\rho(x)}$, where $\psi(x)$ is a subexponential correction and $\rho$ is a norm on $\mathbb{R}^{d}$. Let $\langle\sigma_{0}\sigma_{x}\rangle_{\beta}$ be the usual Ising two-point function with free boundary conditions at inverse temperature $\beta$ without an external field, and $\nu_{\beta}(\hat{x})$ be the associated inverse correlation length in the direction $\hat{x}=x/\norm{x}$ where $\norm{\cdot}$ is the euclidian norm. It is easy to see that one always has $\nu_{\beta}(\hat{x})\leq\rho(x)$. In~\cite{Aoun+Ioffe+Ott+Velenik-CMP-2021}, we developed an explicit necessary and sufficient condition (see Theorem~\ref{thm:explicit_saturation}) to ensure the existence of a non-trivial \textit{saturation transition}, i.e. the strict positivity of $\betasat(\hat{x})=\sup\lbrace \beta\geq 0: \nu_{\beta}(\hat{x})=\rho(\hat{x})\rbrace$. For instance, a sufficient condition for the latter to happen is to have $\psi(x)=\sfO(\norm{x}^{-(d+\varepsilon)})$ for some $\varepsilon>0$. By definition one always has $\betasat(\hat{x})\leq\betac$ where $\betac$ is the usual transition point of the Ising model. Note that if $\betasat(\hat{x})>0$,  the function $\beta\mapsto\nu_{\beta}(\hat{x})$ is non-analytic. Moreover, we proved in~\cite{Aoun+Ioffe+Ott+Velenik-CMP-2021} that if $\betasat(\hat{x})>0$,  then the Ornstein-Zernike asymptotics (see~\eqref{eq:OZ_behaviour}) \textit{do not hold} at arbitrarily high temperature. In subsequent works~\cite{Aoun+Ott+Velenik-2021,Aoun+Ott+Velenik-2022}, we studied the behavior of the two point function in the saturated regime $(0,\betasat(\hat{x}))$ and in the non-saturated regime $(\betasat(\hat{x}),\betac)$. Under appropriate assumptions, for $\beta\in (\betasat(\hat{x}),\betac)$, we proved in~\cite{Aoun+Ott+Velenik-2021} that the two-point function has the \textit{Ornstein-Zernike} asymptotics: there exists $c:=c(\hat{x}, \beta)>0$ such that 
\begin{equation}\label{eq:OZ_behaviour}
\langle\sigma_{0}\sigma_{x}\rangle_{\beta}=c\norm{x}^{-\frac{d-1}{2}}\sfe^{-\nu_{\beta}(x)}(1+\sfo_{\norm{x}}(1)).
\end{equation}
The OZ asymptotics were predicted in the physics literature in~\cite{Ornstein+Zernike-1914}, and were expected to hold generally when the interactions decay exponentially fast in the distance. In~\cite{Aoun+Ott+Velenik-2022}, we proved that this is not the case in the whole saturated regime: under approriate assumptions, for $\beta\in (0,\betasat(\hat{x}))$, there exists $C(\beta,\hat{x})>0$ such that
\begin{equation}\label{eq:1jump_behaviour}
\langle\sigma_{0}\sigma_{x}\rangle_{\beta}=CJ_{x}(1+\sfo_{\norm{x}}(1)).
\end{equation}
This leaves us with a natural question of determining the asymptotics of the two-point function at $\betasat(\hat{x})$. The techniques used for proving~\eqref{eq:1jump_behaviour} and~\eqref{eq:OZ_behaviour} break down at $\betasat(\hat{x})$. On the one hand, in~\cite{Aoun+Ott+Velenik-2021}, we derived~\eqref{eq:OZ_behaviour} under the mass-gap assumption $\nu_{\beta}(\hat{x})<\rho(\hat{x})$, which is violated at $\betasat(\hat{x})$ since by continuity of the function $\beta\mapsto\nu_{\beta}(\hat{x})$, one has $\nu_{\beta}(\hat{x})=\rho(\hat{x})$. On the other hand, we used differential inequalities (inspired by the ideas of~\cite{Duminil-Copin+Raoufi+Tassion-2017, Hutchcroft-2020}) and the fact that for any $\beta_{0}\in (0,\betasat(\hat{x}))$ there exists an open interval containing $\beta_{0}$ on which the function $\beta\mapsto\nu_{\beta}(\hat{x})$ is constant to derive~\eqref{eq:1jump_behaviour}. In the present article, we provide partial answers for the behavior of the two-point function at $\betasat(\hat{x})$: under suitable assumptions, we prove that the Laplace transform associated to the two-point function is infinite. Moreover, we prove that~\eqref{eq:OZ_behaviour} holds up to multiplicative constants on $\bbZ$. This is the first example where the OZ asymptotics are shown to hold in the absence of a mass-gap. In particular, it shows that the mass-gap is not a necessary condition for OZ asymptotics to hold.

Note that in the discussion above, the saturation phenomenon is only shown to happen at high temperatures. In the present work, we prove the existence of a non-trivial saturation regime at arbitrarily \textit{low temperatures} as well. Let $\langle \sigma_{0};\sigma_{x}\rangle_{\beta}$ be the truncated two-point function of the Ising model with $+$ boundary conditions and $\nu_{\beta}(\hat{x})$ the associated inverse correlation length. We prove the existence of $\betasat^{*}:=\betasat^{*}(\hat{x})<\infty$ such that for every $\beta>\betasat^{*}$, we have $\nu_{\beta}(\hat{x})=\rho(\hat{x})$.
\section{Models and notations}

\subsection{Graphs}

Most of our results naturally extend to a wider set-up but we restrict attention to \(\Zd\). We will always see \(\Zd\) as canonically embedded inside \(\bbR^d\) and will denote \(\norm{\cdot}\) the Euclidean norm on \(\Rd\). \(\rho\) will denote a norm on \(\bbR^d\) (and will be one of the parameters in our analysis).

We consider the graph \((\Zd,E_d)\) with edge set \(E_d = \bigl\{\{i,j\}\subset\Zd\bigr\}\), which we will often write simply \(\Zd\). Let \(\Lambda_N=\{-N,\dots,N\}^d\) and \(\Lambda_N(x) = x+\Lambda_N\).

Given a subgraph $\Lambda$, let $\Lambda^{c}=\bbZ^{d}\backslash\Lambda$ and 
\begin{equation*}
	E_\Lambda = \bsetof{\{i,j\}\in E_{d}}{\{i,j\}\subset \Lambda}.
\end{equation*}

Given $x,y,z\in\bbZ^{d}$, a sequence $\gamma=(\gamma_{0},\gamma_{1},\dots,\gamma_{n})\in (\bbZ^{d})^{n+1}$ is called a path from $x$ to $y$ if $\gamma_{0}=x$ and $\gamma_{n}=y$. We say that $n$ is the length of the path, and denote it by $\abs{\gamma}$. We say that $\gamma$ is edge self-avoiding if $\lbrace \gamma_{i},\gamma_{i+1}\rbrace=\lbrace \gamma_{j},\gamma_{j+1}\rbrace\Rightarrow i=j$. 
\subsection{Interaction}
We consider a weight function (the \textit{interaction}, or the set of \textit{coupling constants}) \(J:E_d\to\R_+\) of the form  $J_{i,j}=\psi(i-j)\sfe^{-\rho(i-j)}$ where $\psi$ satisfies
	\begin{equation*}
	\lim\limits_{\norm{x}\rightarrow\infty}\dfrac{\log(\psi(x))}{\norm{x}}=0.
	\end{equation*}
	Moreover, we will assume that the interaction satifies the following properties:
\begin{itemize}
	\item \textit{No self-interaction:} \(J_0=0\),
	\item \textit{Rotational invariance:} $J$ is invariant by a rotation of $\pi/2$ around any coordinate axis.
	\end{itemize}
\subsection{Percolation configurations}
Given a subset $\Lambda$ of $\bbZ^{d}$, the percolation configuration $\omega$ is defined as a function from $E_{\Lambda}$ to $\lbrace 0,1\rbrace$. Given an edge $\lbrace i,j\rbrace\in E_{\Lambda}$, we say that the edge $\lbrace i,j\rbrace$ is open in $\omega$ if $\omega_{i,j}=1$ and closed otherwise. Given the subsets $A,B,C$ of $\Lambda$, we will denote by $\lbrace A\overset{C}{\leftrightarrow}B\rbrace$ the subset of percolation configurations $\omega$ such that there exists a path from $A$ to $B$ consisting of open edges of $C$.  If $C=\Lambda$, we will remove it from the notation. We will write $\lbrace x\overset{C}{\leftrightarrow}y\rbrace$ instead of $\lbrace \lbrace x\rbrace\overset{C}{\leftrightarrow}\lbrace y\rbrace\rbrace$. Finally, we will define the connected component of $x$ by $\calC_{x}:=\lbrace y\in\bbZ^{d}: x\leftrightarrow y\rbrace$.

\subsection{Constants}
\(c, C, c', C', \dots\) will denote constants whose value can change from line to line. Unless explicitly stated otherwise, they depend only on \(d, \beta, h,  J\).

\subsection{Ising Model}
The Ising model at inverse temperature \(\beta\geq 0\) without a magnetic field with free boundary condition on \(\bbZ^d\) is the probability measure on \(\Omega :=\{-1,+1\}^{\bbZ^d}\) given by the weak limit of the finite-volume measures (for \(\sigma\in\{-1,+1\}^{\Lambda_N}\) and \(\Lambda_N=[-N,N]^{d}\cap\mathbb{Z}^{d}\))
\[
	\mu^{\f}_{\Lambda_N;\beta}(\sigma) = \frac{1}{Z^{\f}_{\Lambda_N;\beta}} \sfe^{-\beta\mathscr{H}^{\f}_{N}(\sigma)},
\]
with Hamiltonian
\[
	\mathscr{H}^{\f}_{N}(\sigma) = -\sum_{\{i,j\}\subset\Lambda_N } J_{ij} \sigma_i\sigma_j - \sum_{i\in\Lambda_N}\sigma_i
\]
and partition function \(Z^{\f}_{\Lambda_N;\beta,}\). We also define the Ising measure at inverse temperature $\beta\geq 0$ with $+$ boundary condition and without a magnetic field by 
$$\mu^{+}_{\Lambda_N;\beta}(\sigma) = \frac{1}{Z^{+}_{\Lambda_N;\beta}} \sfe^{-\beta\mathscr{H}^{+}_{N}(\sigma)},$$ with Hamiltonian

$$\mathscr{H}_{N,h}^{+}(\sigma) =  -\sum_{\{i,j\}\subset\Lambda_N } J_{ij} \sigma_i\sigma_j -\sum_{\substack{i \in \Lambda_N \\ j\in \Omega\setminus \Lambda_N}}J_{ij}\sigma_i.$$
For $\eta\in \lbrace +, f \rbrace$, the limit \(\mu^{\eta}_{\beta}=\lim_{N\to\infty}\mu^{\eta}_{\Lambda_N;\beta}\) is always well defined and agrees with the unique infinite-volume measure whenever \(\beta<\betac\), the critical point of the model; we refer to~\cite{Friedli+Velenik-2017} for more details. We will be interested in the behaviour of the \textit{truncated two-point} function of the model 
\begin{equation*}
\langle\sigma_{0};\sigma_{x}\rangle_{\beta}:=\Cov(\sigma_{0},\sigma_{x}),
\end{equation*}
where the covariance is taken with respect to $\mu^{+}_{\beta}$. We also introduce the \textit{correlation length} associated to the latter in the direction $s\in\mathbb{S}^{d-1}$
\begin{equation*}
\nu_{\beta}(s):= -\lim_{n\to\infty} \frac{1}{n} \log \langle\sigma_{0};\sigma_{ns}\rangle_{\beta}.
\end{equation*}
The existence of this limit follows from the subadditivity proved in~\cite{Graham-1982}. The subaddivity also provides the following bound
\begin{equation}\label{ineq:subaddivitiy}
\langle\sigma_{0};\sigma_{ns}\rangle_{\beta}
\leq\sfe^{-n\nu_{\beta}(s)}.
\end{equation}
Let us also introduce
\begin{equation*}
\langle\sigma_{0}\sigma_{x}\rangle_{\beta}
:=
\bbE[\sigma_{0}\sigma_{x}].
\end{equation*} 
When $\beta<\betac$, the truncated two-point function is just equal to the usual two-point function,
\begin{equation*}
\langle\sigma_{0};\sigma_{x}\rangle_{\beta}=\langle\sigma_{0}\sigma_{x}\rangle_{\beta}.
\end{equation*}
The following result was proved in~\cite{Aoun+Ott+Velenik-2022}.
\begin{theorem}\label{thm:uniform_expo_decay}
Fix $\beta<\betac$ and $s\in\bbS^{d-1}$. Then $\nu_{\beta}(s)>0$.  
\end{theorem}
\subsubsection{FK-Ising model}
Intimately related to the Ising model is the FK-Ising model (i.e. the Random-Cluster model with $q=2$). The latter is a measure on percolation configurations on $\bbZ^{d}$ depending on a parameter \(\beta\in\bbR_{\geq 0}\) that will be denoted by $\Phi_{\beta}$ and is obtained as the weak limit of the finite-volume measures
\begin{equation}
	\Phi_{\Lambda_N;\beta}(\omega) = \frac{1}{Z^{\text{FK}}_{\Lambda_N;\beta}} \prod_{\{i,j\}\in\omega}(\sfe^{\beta J_{ij}}-1) 2^{\kappa(\omega)},
\end{equation}
where \(\kappa(\omega)\) is the number of connected components in the graph with vertex set \(\Lambda_N\) and edge set \(\omega\) and \(Z^{\text{FK}}_{\Lambda_N;\beta}\) is the partition function. One has the following correspondance between the Ising model without a magnetic field and the FK-Ising model
\begin{equation*}
	\label{eq:Potts_FK_Corresp}
	\langle\sigma_{0}\sigma_{x}\rangle_{\beta}= \Phi_{\beta}(0\leftrightarrow x).
\end{equation*}
It is a standard consequence that one in particular has 
\begin{equation}\label{ineq:truncated>FK}
\trn_{\beta}\geq\Phi_{\beta}(0\leftrightarrow x, 0\nleftrightarrow\infty).
\end{equation}
During the proofs, we will need several well-known properties of the FK-Ising model:

\vspace{2mm}
\textit{Finite energy property:} Fix $\Lambda\subset\bbZ^{d}$ and $\beta>0$. For any $e\in E_{\Lambda}$ and $\eta\in\lbrace 0,1\rbrace^{E_{\Lambda}\backslash\lbrace e\rbrace}$, one has
\begin{equation}\label{eq:finite_energy}
\dfrac{\sfe^{\beta J_{e}}-1}{\sfe^{\beta J_{e}}+1}
\leq 
\Phi_{\Lambda;\beta}(\omega_{e}=1\hphantom{,}\vert \omega_{E_{\Lambda}\backslash\lbrace e\rbrace}=\eta)\leq 1-\sfe^{-\beta J_{e}}.
\end{equation}

\vspace{2mm}
\textit{FKG inequality:} We say that a $\calF_{\Lambda}$-measurable event $A$ is increasing if $\mathds{1}_{A}$ in increasing with respect to the lexographical order on $\lbrace 0,1\rbrace^{E_{\Lambda}}$. Given two increasing events $A,B$, the FKG inequality states that
\begin{equation*}
\Phi_{\Lambda;\beta}(A)\Phi_{\Lambda;\beta}(B)\leq\Phi_{\Lambda;\beta}(A\cap B).
\end{equation*}

\vspace{2mm}
\textit{Simon--Lieb inequality:} Given a finite subset $S$ containing $0$, one has~\cite{Duminil-Copin+Tassion-2016}
\begin{equation}\label{ineq:Simon-Lieb}
\Phi_\beta(u\xleftrightarrow{\Lambda} v)
	\leq \\
	\sum_{x\in S} \sum_{y\notin S}\Phi_\beta(u \xleftrightarrow{S} x) \beta J_{x,y} \Phi(y\xleftrightarrow{\Lambda} v).
\end{equation}
These properties in particular imply the existence of $C:=C(\beta)>0$ such that 
\begin{equation}\label{ineq:lower_bound_truncated}
CJ_{x}\leq \trn_{\Lambda,\beta}.
\end{equation}
Indeed, one has
\begin{equation*}
\Phi_{\Lambda;\beta}(\omega_{\lbrace 0,x\rbrace}=1, \abs{\calC_{0}}=2)\leq
\Phi_{\beta}(0\leftrightarrow x, 0\nleftrightarrow\infty).
\end{equation*}
The finite energy property then implies the existence of $C:=C(\beta)>0$ such that
\begin{equation*}
CJ_{x}\leq\Phi_{\Lambda;\beta}(\omega_{\lbrace 0,x\rbrace}=1, \abs{\calC_{0}}=2).
\end{equation*}
All these inequalities combined with~\eqref{ineq:truncated>FK} gives~\eqref{ineq:lower_bound_truncated}. Notice that~\eqref{ineq:lower_bound_truncated} in particular implies that $\nu_{\beta}(s)\leq\rho(s)$ for any $\beta>0$.
\subsubsection{Random current}\label{subsection:random_current}
Let $\Lambda$ be a finite subgraph of $\Z^d$.  We consider an additional vertex $\mathfrak{g}$ in the graph $\Lambda$ and denote by $\Lambda^{\mathfrak{g}}$ the graph obtained by adding an edge between each $x\in\Lambda$ and $\mathfrak{g}$. 
A \textit{current} $\mathbf{n}=(\n_{xy})_{x,y\in E_{\Lambda^{\mathfrak{g}}}}$ on $\Lambda^{\mathfrak{g}}$ is an element of $\mathbb{N}^{E_{\Lambda^{\mathfrak{g}}}}$.  For $x\in \Lambda^{\mathfrak{g}}$,  set $X(\mathbf{n},x):=\sum_{y\in \Lambda^{\mathfrak{g}}} \mathbf{n}_{xy}$.  We define $$\partial \mathbf{n} := \lbrace x\in \Lambda^{\mathfrak{g}} : X(\mathbf{n},x) \text{ is odd}\rbrace.$$
In the case of the Ising model on a finite box $\Lambda\subset \mathbb{Z}^d$ with $+$ boundary condition, we set $J_{x\mathfrak{g}} :=\sum_{y\in\Lambda^c}J_{xy}$. This will allow us to reinterpret the $+$ boundary conditions as the presence of a new vertex, namely $\mathfrak{g}$. We also define the weight of a current $\mathbf{n}$ on $\Lambda^{\g}$ to be the quantity $$w_{\Lambda^{\g};\beta}(\n) := \prod_{xy \in E_{\Lambda^{\g}}} \frac{(\beta J_{xy})^{\n_{xy}}}{\n_{xy}!}.$$
Taylor-expanding $\sfe^{\beta J_{ij} \sigma_i \sigma_j}$ and resumming, one gets
$$\langle\sigma_{A}\rangle_{\Lambda,\beta} =\left\{
    \begin{array}{ll}
        \frac{Z_{\Lambda}(A)}{Z_{\Lambda}(\emptyset)} & \mbox{if}~\vert A \vert ~\mbox{is even} \ \\
        \frac{Z_{\Lambda}(A\cup \lbrace g\rbrace)}{Z_{\Lambda}(\emptyset)} & \mbox{otherwise}
    \end{array}
\right.
$$ where $Z_{\Lambda}(F) := \sum_{\n : \partial \n = F} w_{\Lambda;\beta}(\n)$ for any subset $F\subset \Lambda$. We will refer to this correspondance as the \textit{random-current representation}. 
Given a subset $A\subset \Lambda^{\g}$, one can define a probability law on currents on $\Lambda^{\g}$ with sources $A$ by 
$$\p^{A}_{\Lambda^{\g}; \beta}(\n)=\frac{w_{\Lambda^{\g};\beta}(\n)\mathsf{1}_{\partial \n = A}}{Z_{\Lambda}(A)}.$$
We will use the notation $ \p_{\Lambda^{\g};\beta}^{\emptyset,\lbrace 0,x\rbrace}$ for the product measure $\p_{\Lambda^{\g}, \beta}^{\emptyset}\times \p_{\Lambda^{\g},\beta}^{\lbrace 0,x\rbrace}$. This is therefore a law on pairs of currents $(\n_1,\n_2)$ such that $\partial\n_1 = \emptyset$ and $\partial \n_2=\lbrace 0,x\rbrace$.  In particular,  $0$ and $x$ are connected in $\n_2$ since those are the only vertex with odd degree.  Such a pair $\n=(\n_1,\n_2)$ can be seen as the sum $\n_1 + \n_2$.  
It is well known (see for instance~\cite{Duminil-Copin-2016} that
\begin{equation}\label{eq:truncated_to_current}
\trn_{\Lambda; \beta} = \langle \sigma_0 \sigma_x \rangle_{\Lambda;\beta}\p^{\emptyset, \lbrace 0,x\rbrace}_{\Lambda^{\g};\beta}\left[0\nleftrightarrow\mathfrak{g} \right].
\end{equation}
Note that every current $\n$ can be seen as a percolation realization $(\omega_e)_{e\in E_{\Lambda}}$,  by declaring an edge $e$ is said to be open if and only if $\mathbf{n}_e > 0$. 
\subsubsection*{Partial finite energy property}
One can show that~\cite{Ott-2020}
\begin{equation}\label{ineq:finite_energy_lower_bound}
 \p^{A}_{\Lambda^{\g},\beta} \left(\n_e >0~\vert~  \n_f = m(f),~\forall f\neq e \right)\geq \frac{\cosh (\beta J_{e})-1}{\cosh(\beta J_{e})},
\end{equation} 
for any edge $e=\lbrace a,b\rbrace\in E_{\Lambda}$ and any function $m: \Lambda\setminus \lbrace e \rbrace\rightarrow \mathbb{N}$  compatible with $A$. This in particular implies that 
 \begin{equation}\label{ineq:useful_closing}
 \p^{A}_{\Lambda^{\g};\beta} \left(\n_e =0~\vert~  \n_f = m(f),~\forall f\neq e \right)\leq 2 \sfe^{-\beta J_{e}}.
 \end{equation}
Furthermore, recall that if $\beta>\beta_c$, then for any set $B$ with $\abs{B}\in\lbrace 0,2\rbrace$, there exists $C'>0$ such that $C'\leq \langle \sigma_B \rangle_{\Lambda ; \beta,h} \leq 1$ (we set $\sigma_{\emptyset}=1$). There exists $C>0$ such that for any set  $\lbrace e_1,...,e_k\rbrace$ of edges in $\Lambda$,  one has
$$\p^{B}_{\Lambda^{\g};\beta,h}\left(\n_{e_1} \geq 1,..., \n_{e_k}\geq 1\right )\leq  C\beta^k \prod_{i=1}^{k}J_{e_i}.$$
Indeed, summing on all currents $\n$ with $\partial\n=B$ satisfying $\n_{e_i} \geq 1$ for $1\leq i\leq k$,  one gets
\begin{align*}
\sum_{\n}  \frac{w(\n)}{Z_{\Lambda}(B)} &\leq \left(\beta^k \prod_{i=1}^{k}J_{e_i} \right)\sum_{\tilde{\n}} \frac{w(\tilde{\n})}{Z_{\Lambda}(B)}\leq 
 \frac{\langle\sigma_{S}\rangle_{\Lambda;\beta}}{\langle \sigma_{B}\rangle_{\Lambda;\beta}}\beta^{k}\prod_{i=1}^{k}J_{e_i}\leq C\beta^{k}\prod_{i=1}^{k}J_{e_i},
\end{align*}
where the second sum is on the currents $\tilde{\n}$ having as sources set the symmetric difference $S:=A\Delta e_1 \Delta...\Delta e_k$.
Putting these two results together, one thus gets 
\begin{equation}\label{ineq:like_finite_energy}
\p^{B}_{\Lambda^{\g};\beta} \left(\n_{e_1}\geq 1,...,\n_{e_k}\geq 1, \n_{f_1}=...=\n_{f_l}=0\right) \leq C 2^l \beta^k \prod_{i=1}^{k}J_{e_i}\prod_{j=1}^{l}\sfe^{-\beta J_{f_j}},
\end{equation}
for any family $\lbrace f_1,...,f_l\rbrace$ of edges.

\subsubsection{Convex geometry}
It will be convenient to introduce a few quantities associated to the norm $\rho$. First, two convex sets are important: the unit ball $\mathscr{U}\subset \bbR^d\) associated to $\rho$ and the corresponding
 \emph{Wulff shape}
\[
	\Wulff = \setof{t\in\bbR^d}{\forall x\in\bbR^d,\, t\cdot x \leq \rho(x)}.
\]
Given a direction \(s\in \bbS^{d-1}\), we say that the vector \(t\in\bbR^d\) is dual to \(s\) if \(t\in\partial\Wulff\) and \(t\cdot s = \rho(s)\). A direction \(s\) possesses a unique dual vector \(t\) if and only if \(\Wulff\) does not possess a facet with normal \(s\). Equivalently, there is a unique dual vector when the unit ball $\mathscr{U}$ has a unique supporting hyperplane at $s/\rho(s)$. (See Fig.~\ref{fig:duality} for an illustration.) We refer to~\cite{Ioffe-2015} for the necessary backround on the convex geometry.

\begin{figure}[ht]	\includegraphics{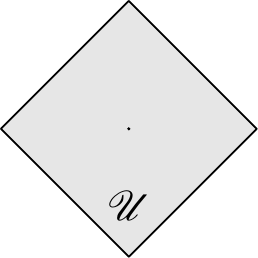}
	\hspace*{1cm}
	\includegraphics{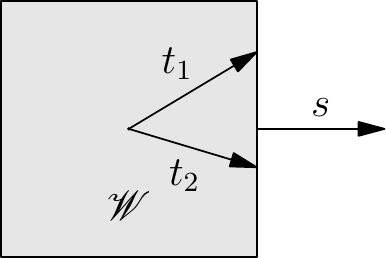}
	\hspace*{1cm}
	\includegraphics{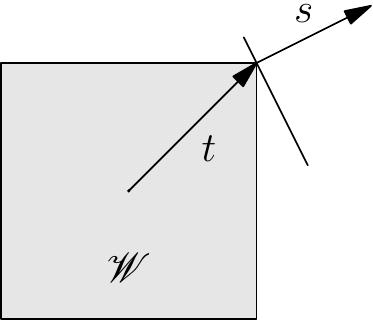}
	\caption{Left: The unit ball for the norm $\rho(\cdot)=\normI{\cdot}\). Middle: the corresponding Wulff shape $\Wulff$ with two vectors \(t_1\) and \(t_2\) dual to \(s=(1,0)\). Right: the set \(\Wulff\) with the unique vector \(t\) dual to \(s=\frac{1}{\sqrt{5}}(2,1)\).}
	\label{fig:duality}
\end{figure}
\subsubsection{Saturation transition}

Recall that~\eqref{ineq:lower_bound_truncated} implies that $\nu_{\beta}(s)\leq\rho(s)$ for every $\beta>0$.  As explained in the introduction, we consider the saturation point above the critical temperature in the direction $s\in\bbS^{d-1}$ defined by
\begin{equation*}
\betasat(s)=\sup\lbrace \beta\in [0,\betac]: \nu_{\beta}(s)=\rho(s)\rbrace.
\end{equation*}
For $t\in\Wulff$, we define
\begin{equation*}
\bbG_{\beta}(t)=\sum_{x\in\mathbb{Z}^{d}}\sfe^{t\cdot x}\Phi_{\beta}(0\leftrightarrow x)\qquad \text{ and }\qquad \bbJ(t)=\sum_{x\in\mathbb{Z}^{d}}\sfe^{t\cdot x}J_{0,x},
\end{equation*}
and an associated transition points
\[
	\betahat(t) = \sup\setof{\beta\geq 0}{\bbG_{\beta}(t) < \infty},
\]
and 
\begin{equation*}
\betahat(s)=\smashoperator{\sup_{\substack{t\in\Wulff \\ \text{t dual to s}}}}\betahat(t).
\end{equation*}
It was proved in~\cite{Aoun+Ott+Velenik-2022} that if $\psi(x)=\rho(x)^{-\alpha}$ with $\alpha>2d$ or $\psi(x)=\sfe^{-c\rho(x)^{\eta}}$ with $\eta\in(0,1)$ and $c>0$, then $\betahat(s)=\betasat(s)$. We can now state the criterion ensuring the existence of a non-trivial saturation point:
\begin{theorem}\label{thm:explicit_saturation}
Let $J$ be exponentially decaying. Fix $s\in\bbS^{d-1}$.Then $\betasat(s)>0$ if and only if there exists a dual vector $t$ to $s$ such that $\bbJ(t)<\infty$.
\end{theorem}
Note that $\bbJ(t)<\infty$ whenever $\psi(x)=\sfO(\rho(x)^{-d-\varepsilon})$ for some $\varepsilon>0$. An even more explicit (although a little bit less general) criterion ensuring the finitude of $\bbJ(t)$ was derived in~\cite{Aoun+Ioffe+Ott+Velenik-CMP-2021}. It was proved in~\cite{Duminil-Copin+Tassion-2016} that $\nu_{\betac}(s)=0$ for every $s\in\bbS^{d-1}$, and therefore one always has $\betasat(s)<\betac$ by the continuity of the function $\beta\mapsto\nu_{\beta}(s)$.

We also introduce a saturation point below the critical temperature in the direciton $s$ defined by 
\begin{equation*}
\betasat^{*}(s)=\sup\lbrace \beta\in [\betac,\infty): \nu_{\beta'}(s)=\rho(s)\hphantom{,} \forall\beta'>\beta\rbrace.
\end{equation*}

\section{Main results and conjectures}

\begin{theorem}\label{lemma:transformée_infini_point_satuation}
	For any $t\in\Wulff$ such that $\betahat(t)>0$, there exists $C:=C(\betahat(t))$ and a strictly increasing sequence $(n_{k})_{k=1}^{\infty}$ such that 
	\[
		\sum_{x\in\Lambda_{n_{k}}} \sfe^{t\cdot x} \Phi_{\betahat(t)}(0 \xleftrightarrow{\Lambda_{n_{k}}} x)\geq Ck.
	\]
In particular,
	\begin{equation*}
	\bbG_{\betahat(t)}(t)=\infty.
	\end{equation*}
	Moreover, if $\psi(x)=\sfO(\rho(x)^{-d-1-\varepsilon})$ for some $\varepsilon>0$, one can choose $n_{k}=k$. 
\end{theorem}
Theorem~\ref{lemma:transformée_infini_point_satuation} has the following immediate Corollary.
\begin{corollary}\label{corollary:laplace_infinite_beta_sat}
Suppose that $\psi$ has one of the following forms:
\begin{itemize}
		\item \(\psi(x) = \rho(x)^{-\alpha}\) with $\alpha>2d$
		\item \(\psi(x) = \sfe^{-\tilde{c}\rho(x)^{\eta}}\) with \(\tilde{c}>0\) and \(\eta\in (0,1)\).
	\end{itemize}
	Then there exists $C>0$ such that for any $t$ dual to $s$
	\begin{equation*}
		\sum_{x\in\Lambda_{n}} \sfe^{t\cdot x} \Phi_{\betasat(s)}(0 \xleftrightarrow{\Lambda_{n}} x)\geq Cn.
	\end{equation*}
	In particular, $\bbG_{\betasat(s)}(t)=\infty$ for any $t$ dual to $s$.
\end{corollary}
\begin{proof}
It was proved in~\cite{Aoun+Ott+Velenik-2022} that $\betasat(s)=\betahat(s)$ under the assumptions of Corollary~\ref{corollary:laplace_infinite_beta_sat}. Therefore, the conclusion follows from Theorem~\ref{lemma:transformée_infini_point_satuation} since $\betahat(t)\leq\betahat(s)$.

\end{proof}

The next result will give a description of the saturation phenomenon as a function of the direction $s$: if $\Wulff$ is regular locally in a strictly saturated direction $s$ (in the sense that $\beta<\betasat(s)$), then there exists a neighborhood of $s$ for which all the directions are strictly saturated.
\begin{lemma}\label{lemma:local_saturation}
Fix $t\in\Wulff$ and assume that $\Wulff$ is locally strictly convex and $C^{1}$. Fix the unique direction $s\in\bbS^{d-1}$ dual to $t$. Assume that $\betasat(s)=\betahat(s)$ locally and that there exists \(\delta>0\) such that \(\bbJ(h)<\infty\) for all \(h\in\partial\Wulff\cap B_{\delta}(t)\). Then, for every \(\beta<\betasat(s)\), there exists \(\varepsilon>0\) such that for any $s'\in\bbS^{d-1}\cap B_{\varepsilon}(s)$, $\beta<\betasat(s')$.
\end{lemma} 
The next result gives the asymptotics of the two-point function at $\betasat(1)$ on $\bbZ$.
\begin{theorem}\label{thm:OZ_sur_Z}
Fix $d=1$. Suppose that
\begin{itemize}
		\item \(\psi(x) = \abs{x}^{-\alpha}\) with $\alpha>2$
		\item \(\psi(x) = \sfe^{-\tilde{c}\abs{x}^{\eta}}\) with \(\tilde{c}>0\) and \(\eta\in (0,1)\).
	\end{itemize}
Then, there exists $C_{-}>0$ such that for any $x\in\bbZ$, one has
\begin{equation*}
C_{-}\leq \sfe^{\rho(x)}\Phi_{\betasat(1)}(0\leftrightarrow x)\leq 1.
\end{equation*}
\end{theorem}

Our next result shows that a non-trivial saturation regime can exist even at arbitrarily low temperatures for the truncated two-point function.
\begin{theorem}\label{thm:saturation_truncated}
Let $d\geq 2$ and $s\in\bbS^{d-1}$.  Suppose that $J_{e}>0$ for any edge of length 1. If there exists $t\in\partial\calW$ dual to $s$ such that $\bbJ(t)<\infty$, then there exists $\beta_{0}$ such that $\betasat^{*}(s)<\beta_{0}$. Moreover, for any $\beta>\beta_{0}$ there exists $C_{-},C_{+}>0$ such that
\begin{equation*}
C_{-}J_{ns}\leq \langle\sigma_{0};\sigma_{ns}\rangle_{\beta}\leq C_{+}J_{ns}.
\end{equation*}
\end{theorem}
\begin{remark}
Notice that in Theorem~\ref{thm:saturation_truncated}, we take $d\geq 2$. This assumption is necessary, since by definition one has $\betasat^{*}(s)\geq\betac$, and $\betac=\infty$ on $\bbZ$. This is in contrast with what happens at high temperatures, in which case Theorem~\ref{thm:explicit_saturation} holds.
\end{remark}
Theorem~\ref{thm:saturation_truncated} is in contrast with what happens in the finite-range Ising model, in which case it was proved in~\cite{Bricmont+Frohlich-1985b} that the truncated two-point function satisfies OZ asymptotics on $\bbZ^{d}$ with $d\geq 3$ (see also~\cite{Campanino+Gianfelice-2015}).

Our work suggests a number of conjectures and open problems that we summarize now. 

\subsubsection{Behaviour at $\betasat(s)$}

Theorem~\ref{thm:OZ_sur_Z} suggests that the OZ asymptotics should hold at $\betasat(s)$ whenever $\psi$ decays fast enough.
\begin{conjecture}
For any $\psi$ decaying fast enough, the conclusion of Theorem~\ref{thm:OZ_sur_Z} holds on $\bbZ^{d}$.
\end{conjecture}
However, this is easily seen not to be true in general. To see that, fix $\rho(\cdot)=\norm{\cdot}_{1}$. Using the results of~\cite{Aoun+Ioffe+Ott+Velenik-CMP-2021}, it can easily be seen that for $\psi(x)=\rho(x)^{-\alpha}$, $\betasat(e_{1})>0$ whenever $\alpha>1$. However, one always has the lower bound
\begin{equation*}
CJ_{x}\leq\langle\sigma_{0}\sigma_{x}\rangle_{\beta}.
\end{equation*}
This shows that OZ asymptotics cannot hold in this case whenever $d>3$.
\begin{openproblem}
Caracterise all possible behaviours of the two-point function at $\betasat(s)$ in function of the dimension and $\rho$.
\end{openproblem}
We expect that the OZ asymptotics could fail at $\betasat$ for two different reasons:
\begin{enumerate}
\item The dominant contribution to the FK-Ising two-point function comes from configurations with $\abs{\calC_{0}}=\sfo(n)$.

\item The dominant contribution to the FK-Ising two-point function comes from configurations with $\abs{\calC_{0}}=\sfO(n)$ (as is the case in the OZ regime), but the steps of the associated effective random walk don't have two moments, and so the usual local limit theorem does not hold (however, there has been results on the non-OZ asymptotic behaviour of the Green function in this case, see~\cite{Berger-2020} and references therein).
\end{enumerate}
We plan to come back to this issue in a simpler context of the killed random walk (see section~\ref{section:saturation} for the definition of this model).

\subsubsection{Behaviour for $\beta>\betac$}
In the case of exponentially decaying coupling constants, Theorem~\ref{thm:saturation_truncated} implies the exponential decay of the two-point function for $\beta$ large enough whenever $\psi$ decays fast enough. We expect this to hold more generally below the critical temperature.
\begin{conjecture}\label{conjecture:expo_decay}
If there exists $c>0$ such that $J_{x}\leq \sfe^{-c\norm{x}}$, then $\nu_{\beta}(s)>0$ for $\beta>\betac$ and $s\in\bbS^{d-1}$.
\end{conjecture}
Understanding the behaviour of the truncated two-point function without an external field non pertubatively below critical temperature is challenging. The exponential decay of the two-point function in the finite-range Ising models was established only recently in~\cite{Duminil-Copin+Goswami+Raoufi-2019}. For $\beta<\beta_{c}$, the conclusion of Conjecture~\ref{conjecture:expo_decay} was established in~\cite{Aoun+Ott+Velenik-2022} using the random cluster representation of the Ising model and the OSSS inequality for monotonic measures (see~\cite{Duminil-Copin+Raoufi+Tassion-2017}). Since the double random-current is not known to be monotonic, one cannot use the same reasoning to prove Conjecture~\ref{conjecture:expo_decay}.

We also expect the same dichotomy of behaviour of the two-point function between $\beta<\betasat(s)$ and $\beta\in(\betasat(s),\betac)$ (see~\eqref{eq:1jump_behaviour} and~\eqref{eq:OZ_behaviour}) to happen below the critical temperature.
\begin{conjecture}
\begin{enumerate}
\item For $\beta>\betasat^{*}(s)$, there exists $C>0$ such that 
\begin{equation*}
\langle\sigma_{0};\sigma_{ns}\rangle_{\beta}=CJ_{ns}(1+\sfo_{n}(1)).
\end{equation*}
\item For $\beta\in (\betac,\betasat^{*}(s))$, there exists $C>0$ such that
\begin{equation*}
\langle\sigma_{0};\sigma_{ns}\rangle_{\beta}=Cn^{-\frac{d-1}{2}}\sfe^{-\nu_{\beta}(x)}(1+\sfo_{n}(1)).
\end{equation*}
\end{enumerate}
\end{conjecture}
\subsection{Organisation of the paper}
In Section~\ref{section:phi(S)_argument}, we will prove Theorem~\ref{lemma:transformée_infini_point_satuation}, Theorem~\ref{thm:OZ_sur_Z} and Lemma~\ref{lemma:local_saturation} using the so-called $\varphi(S)$ argument. In Section~\ref{section:saturation}, we will prove Theorem~\ref{thm:saturation_truncated} by comparing directly the random-current representation of the truncated two-point function to the Green function associated to a well-chosen killed random walk. Note that different parts are essentially independent. 

\section{$\varphi(S)$ argument}\label{section:phi(S)_argument}
In this section, we are going to prove Theorem~\ref{lemma:transformée_infini_point_satuation}, Lemma~\ref{lemma:local_saturation} and Theorem~\ref{thm:OZ_sur_Z}. Generalizing what has been done in~\cite{Duminil-Copin+Tassion-2016}, given a finite subset \(S\) containing \(0\), $t\in\partial\Wulff$ and $\beta>0$, let us define 
\[
	\varphi_{\beta}(S,t) = \beta\sum_{x\in S} \sum_{y\notin S} \sfe^{t\cdot x} \Phi_\beta(0\xleftrightarrow{S}x) J_{x,y} \sfe^{t\cdot (y-x)}.
\]
Moreover, we define
\begin{equation*}
	\tilde{\beta}_{\mathrm{sat}}(t) = \sup\setof{\beta\geq 0}{\text{there exists a finite \(S\) containing \(0\) such that \(\varphi_{\beta}(S,t)<1\)}}.
\end{equation*}
We will need the following lemma:
\begin{lemma}\label{lemma:phi(S)_argument}
Fix $t\in\partial\Wulff$ and $\beta>0$. Assume that $\betahat(t)>0$.
\begin{enumerate}
\item If there exists a finite subset $S\ni 0$ such that $\varphi_{\beta}(S,t)<1$, then there exists $C=C(S)>0$ such that 
\begin{equation*}
\bbG(t)\le \dfrac{C}{1-\varphi_{\beta}(S,t)}<\infty.
\end{equation*}
\item There exist $c>0$ and a strictly increasing sequence $(n_{k})$ such that 
\begin{equation*}
c\sum_{k=1}^{l}\varphi_{\beta}(\Lambda_{n_{l}},t)\leq\sum_{x\in\Lambda_{n_{l}}} \sfe^{t\cdot x} \Phi_{\beta}(0 \xleftrightarrow{\Lambda_{n_{l}}} x).
\end{equation*}
If $\psi(x)=\sfO(\rho(x)^{-d-1-\varepsilon})$ for some $\varepsilon>0$, one can take $n_{k}=k$.
\end{enumerate}
In particular, $\tilde{\beta}_{\mathrm{sat}}(t)=\betahat(t)$.
\end{lemma}
Before proving Lemma~\ref{lemma:phi(S)_argument}, let us see how it implies Theorem~\ref{lemma:transformée_infini_point_satuation} and Lemma~\ref{lemma:local_saturation}.
\begin{proof}[Proof of Theorem~\ref{lemma:transformée_infini_point_satuation}]
 Since \(\varphi_{\beta}(S,t)\) is a continuous function in \(\beta\) and \([0,1)\) is open in \([0,\infty)\), it follows that at \(\tilde{\beta}_{\mathrm{sat}}(t)\), for every \(\Lambda \ni 0\), we have \(\varphi_{\tilde{\beta}_{\mathrm{sat}}(t)}(\Lambda,t) \geq 1\). This in turn implies, by the second part of Lemma~\ref{lemma:phi(S)_argument}, the conclusion of Theorem~\ref{lemma:transformée_infini_point_satuation}.
\end{proof}
\begin{proof}[Proof of Theorem~\ref{thm:OZ_sur_Z}]
We will only show the result for $x>0$, since the result follows for $x$ negative by symmetry. The right inequality follows directly from~\eqref{ineq:subaddivitiy} and $\nu_{\betasat(1)}(1)=\rho(1)$. For the left inequality, remark that since we assumed that $\alpha>2$, it follows from Corollary~\ref{corollary:laplace_infinite_beta_sat} that for every $x\geq 1$:  
\begin{equation*}
\sum_{k=1}^{x}\sfe^{k}\Phi_{\betasat(1)}(0\leftrightarrow k)\geq Cx -\sum_{k=0}^{x}\sfe^{-k}\Phi_{\betasat(1)}(0\leftrightarrow -k).
\end{equation*}
Since $\sfe^{k}\Phi_{\betasat(1)}(0\leftrightarrow k)\in [0,1]$ by~\eqref{ineq:subaddivitiy}, this implies that there exists $R>0$ and $c>0$ such that for any $m\in\mathbb{N}$, there exists $k\in\lbrace m,\dots,m+R\rbrace$ such that one has 
\begin{equation*}
\sfe^{k}\Phi_{\betasat(1)}(0\leftrightarrow k)\geq c.
\end{equation*}
The result then follows by the finite energy~\eqref{eq:finite_energy} and FKG for every $x\in\mathbb{N}$. 
\end{proof}
\begin{proof}[Proof of Lemma~\ref{lemma:local_saturation}]
In order to prove Lemma~\ref{lemma:local_saturation}, note that, by assumption, \(\beta<\betasat(s) = \hat{\beta}_{\mathrm{sat}}(t) = \tilde{\beta}_{\mathrm{sat}}(t)\), where the last equality is given by Lemma~\ref{lemma:phi(S)_argument}. It follows that there exists a finite \(S\) containing \(0\) such that \(\varphi_{\beta}(S,t) < 1\). Since \(\bbJ\) is locally finite (around \(t\)) and \(S\) is finite, it follows by continuity that \(\varphi_{\beta}(S,h) < 1\) for \(h\in B_{\varepsilon'}(t)\cap\calW_{\rho}\). This implies that \(\beta<\betasat(s')\) for \(s'\) in some small neighborhood around \(s\) since $\Wulff$ is locally strictly convex and \(\betasat = \tilde{\beta}_{\mathrm{sat}}\) locally, which is the desired result.
\end{proof}
\begin{proof}[Proof of Lemma~\ref{lemma:phi(S)_argument}]
We follow here ideas developed in~\cite{Duminil-Copin+Tassion-2016}.
First, suppose that there exists \(S\) containing \(0\) such that \(\varphi_{\beta}(S,t)<1\). Let \(\Lambda\subset\Zd\) and let
\[
	\tilde{\chi}(\Lambda, t, \beta) = \max\Bsetof{\sum_{v\in\Lambda} \sfe^{t\cdot (v-u)} \Phi_\beta(u \xleftrightarrow{\Lambda} v)}{u\in\Lambda}.
\]
Let us fix \(u\in\Lambda\) and denote by \(S_u\) the translation of \(S\) by \(u\). Fix \(v\in\Lambda\setminus S_u\). If \(u\) is connected to \(v\), then there exists \(x\in S_u\) and \(y\notin S_u\) such that \(u\) is connected to \(x\) in \(S\), \(\{x,y\}\) is open and \(y\) is connected to \(v\). Using the union bound and  the Simon--Lieb inequality~\eqref{ineq:Simon-Lieb}, we get 
\begin{equation*}
	\sfe^{t\cdot (v-u)} \Phi_\beta(u\xleftrightarrow{\Lambda} v)
	\leq \\
	\sum_{x\in S_u} \sum_{y\notin S_u} \sfe^{t\cdot (x-u)} \Phi_\beta(u \xleftrightarrow{S_u} x) \sfe^{t\cdot (y-x)} \beta J_{x,y} \sfe^{t\cdot (v-y)} \Phi(y\xleftrightarrow{\Lambda} v).
	\end{equation*}
Summing over \(v\in\Lambda\setminus S_u\), we get 
\[
	\sum_{v\in\Lambda\setminus S_u} \sfe^{t\cdot (v-u)} \Phi_\beta(u \xleftrightarrow{\Lambda}v ) \leq \varphi_{\beta}(S,t) \tilde{\chi}(\Lambda, t, \beta),
\]
where we used the invariance under translations. Since \(S\) is finite, there exists $C:=C(S)>0$ such that 
\[
	\sum_{v\in\Lambda} \sfe^{t\cdot (v-u)} \Phi_\beta(u \xleftrightarrow{\Lambda} v)
	\leq C + \varphi_{\beta}(S,t) \tilde{\chi}(\Lambda, t, \beta).
\]
Now, we can optimize over \(u\) to get 
\[
	\tilde{\chi}(\Lambda,t,\beta) \leq C + \varphi_{\beta}(S,t) \tilde{\chi}(\Lambda,t,\beta),
\]
which can be rewritten as
\[
	\tilde{\chi}(\Lambda,t,\beta) \leq \frac{C}{1-\varphi_{\beta}(S,t)}.
\]
Taking the limit \(\Lambda\uparrow\Zd\), we obtain
\[
	\bbG_{\beta}(t) \leq \frac{C}{1-\varphi_{\beta}(S,t)} < \infty,
\]
where the last inequality follows from the assumption \(\varphi_{\beta}(S,t) < 1\).

Let us now turn to the second point. For any strictly increasing sequence $(n_{k})$, one has 
\begin{align*}
	\sum_{k=1}^{l} \varphi_{\beta}(\Lambda_{n_{k}},t)
	&=
	\sum_{k=1}^{l} \sum_{x\in\Lambda_{n_{k}}} \sum_{y\notin \Lambda_{n_{k}}} \sfe^{t\cdot x} \Phi_\beta(0 \xleftrightarrow{\Lambda_{n_{k}}} x) \sfe^{t\cdot (y-x)} \beta J_{x,y} \\
	&\leq 
	\sum_{x\in\Lambda_{n_{l}}} \sfe^{t\cdot x}\Phi_\beta(0 \xleftrightarrow{\Lambda_{n_{l}}} x) 
	\sum_{k=1}^{l} \sum_{y\notin \Lambda_{n_{k}}} \sfe^{t\cdot (y-x)} \beta J_{x,y} \mathds{1}_{x\in\Lambda_{n_{k}}}.
\end{align*}
Given \(x\in\Zd\), let us prove that the double sum over \(k\) and \(y\) is finite. The sum over \(y\) is bounded by \(\bbJ(t)$ which is finite. Indeed, for $\beta=\betahat(t)/2 >0$ by hypothesis, one has by finite energy
\begin{equation*}
\bbG_{\beta}(t)\geq C_{\beta}\bbJ(t).
\end{equation*}
This implies that 
\begin{equation*}
\lim_{L\rightarrow\infty}\sum_{y\notin\Lambda_{L}}\sfe^{t\cdot (y-x)}J_{xy}=0.
\end{equation*}
 We can thus choose \(n_{k}\) such that 
 \begin{equation*}
\sum_{k\geq 1} \sum_{y\notin\Lambda_{n_{k}}} \sfe^{t\cdot (y-x)}
	 J_{x,y}=C<\infty.
	 \end{equation*}
Moreover, if $\psi(x)=\sfO(\rho(x)^{-d-1-\varepsilon})$ for some $\varepsilon>0$, this last sum is finite if one chooses $n_{k}=k$ since $t\cdot y-\rho(y)\leq 0$ for any $y\in\bbZ^{d}$ by definition of the dual vector $t$. Therefore, we get 
\[
	\sum_{k=1}^{l} \varphi_{\beta}(\Lambda_{n_{k}},t) \leq C\sum_{x\in\Lambda_{n_{l}}}\sfe^{t\cdot x} \Phi_\beta(0 \xleftrightarrow{\Lambda_{n_{l}}} x),
\]
which proves the desired identity. 
\end{proof}
\section{The existence of a saturation transition at low temperatures}
\label{section:saturation}
In this section, we are going to prove Theorem~\ref{thm:saturation_truncated}. Through this section, we are going to assume that $J_{e}>0$ for any edge of length 1. By rotational invariance, we can assume without loss of generality that $J_{e}=1$ for every edge of length $1$.  We start by making a brief summary of the result proved in~\cite{Aoun+Ioffe+Ott+Velenik-CMP-2021} that we rely on.
Given $\lambda>0$, we define the Green function of the killed random walk model by
\begin{equation*}
G_{\lambda}^{\mathrm{KRW}}(x,y)=
\sum_{\gamma: x\rightarrow y}\prod_{i=1}^{\abs{\gamma}}\lambda J_{\gamma_{i-1},\gamma_{i}},
\end{equation*}
where the sum is over edge self-avoiding paths from $x$ and $y$. 
We will need the following result proved in~\cite{Aoun+Ioffe+Ott+Velenik-CMP-2021}.
\begin{theorem}\label{thm:KRW_saturation}
Fix $s\in\bbS^{d-1}$. If there exists a dual vector $t$ to $s$ such that $\bbJ(t)<\infty$, then there exists $\lambda_{0}$, such that for every $\lambda<\lambda_{0}$, there exists $C:=C(\lambda)>0$ such that 
\begin{equation*}
G^{\mathrm{KRW}}_{\lambda}(0,x)\leq CJ_{x}.
\end{equation*}
\end{theorem}
The next result bounds the truncated two-point function of the Ising model by the Green functions introduced above
\begin{lemma}\label{lemma:saturation_comparaison_KRW}
There exists $\beta_{0}$ such that for any $\beta>\beta_{0}$, there exists $C_{+}>0$ such that 
\begin{equation*}
\langle\sigma_{0};\sigma_{x}\rangle_{\beta}\leq C_{+}G_{\lambda(\beta)}^{\mathrm{KRW}}(0,x),
\end{equation*}
where $\lim_{\beta\rightarrow\infty}\lambda(\beta)=0$.
\end{lemma}
Before proving Lemma~\ref{lemma:saturation_comparaison_KRW}, let us show how it implies Theorem~\ref{thm:saturation_truncated}.
\begin{proof}[Proof of Theorem~\ref{thm:saturation_truncated}]
On the one hand, the lower bound in Theorem~\ref{thm:saturation_truncated} follows directly from~\eqref{ineq:lower_bound_truncated} for any $\beta>0$. Fix $\beta_{1}$ such that $\lambda(\beta)<\lambda_{0}$ for any $\beta>\beta_{1}$. Let $\beta_{2}=\max\lbrace \beta_{0},\beta_{1}\rbrace$. Then, for any $\beta>\beta_{2}$
\begin{equation*}
\langle\sigma_{0};\sigma_{x}\rangle_{\beta}\leq C_{+}G_{\lambda(\beta)}^{\mathrm{KRW}}(0,x)\leq cJ_{x},
\end{equation*}
where we used Lemma~\ref{lemma:saturation_comparaison_KRW} in the first inequality and Theorem~\ref{thm:KRW_saturation} in the second inequality. This gives the desired result.
\end{proof}
\noindent
\textit{Heuristic proof of Lemma \ref{lemma:saturation_comparaison_KRW}}\hphantom{,}
Thanks to~\eqref{eq:truncated_to_current}, we need to compare $\p^{\emptyset,\lbrace 0,x\rbrace}_{\Lambda^{\g},\beta}[0\nleftrightarrow \mathfrak{g}]$ to a Green function of a killed random walk. Recall that one has $J_{x\mathfrak{g}}=\sum_{y\in\Lambda^{c}}J_{xy}$. Therefore, thanks to~\eqref{ineq:useful_closing}, most of the points $y$ close to $\partial\Lambda$ will not be connected to $\partial\Lambda$, which allows us to replace the event $\lbrace 0\nleftrightarrow\mathfrak{g}\rbrace$ with the event $\lbrace 0\nleftrightarrow\partial\Lambda\rbrace$. This term can be estimated using Peierls-like argument: we will decompose $\calC_{0,x}$ into $C_{1},\dots,C_{k}$ where $C_{i}$'s are disjoint nearest neighbor connected components of $\calC_{0,x}$. We will extract a path $\gamma$ from $0$ to $x$ in such a way that all points of $\gamma$ are in $\cup_{i=1}^{k}C_{i}$ and that $\abs{\gamma^{(i)}}=K\abs{\partial C_{i}}$ for some $K>0$ where $\gamma^{(i)}$ is the part of $\gamma$ in $C_{i}$. In this way, using~\eqref{ineq:like_finite_energy} and stantard perturbative estimates, we will extract for $\gamma^{(i)}$ a cost of order 
\begin{equation*}
e^{-c\beta\abs{\partial C_{i}}}\beta^{\abs{\gamma^{(i)}}}\prod_{e\in\gamma^{(i)}}J_{e},
\end{equation*}
which can compared easily to a Green function of a killed random walk with parameter $\lambda(\beta)$ satisfying $\lim_{\beta\rightarrow\infty}\lambda(\beta)=0$.
\begin{proof}[Proof of Lemma~\ref{lemma:saturation_comparaison_KRW}]
We are going to prove Lemma~\ref{lemma:saturation_comparaison_KRW} only for $d=2$ where the use of planar duality simplifies the notations. One can generalize the argument that follows for any $d\geq 2$ in a standard way by introducing $d-1$ dimensional plaquettes (i.e., the $d-1$ dimensional faces of a $d$ dimensional hypercubes). We define the \textit{dual graph} of $\bbZ^2$ by 
\begin{equation*}
(\bbZ^{2})^{\ast}=\bbZ^2 + (1/2,1/2).
\end{equation*}
The edges of $(\bbZ^{2})^{\ast}$ are called \textit{dual edges}, and any dual edge $e^{\ast}$ is perpendicular to an unique edge $e$ of $\bbZ^2$. Therefore, there is a one-to one correspondance between percolation configurations on $\bbZ^2$ and those on $(\bbZ^{2})^{\ast}$,  where a dual edge $e^{\ast}$ is open if and only if $e$ is closed.  

We are going to use the random-current representation of the truncated two-point function~\eqref{eq:truncated_to_current} (see section~\ref{subsection:random_current}). We will only work with a single current since one has 
\begin{equation*}
\p^{\emptyset,\lbrace 0,x\rbrace}_{\Lambda^{\g}_N,\beta}[0\nleftrightarrow \mathfrak{g}] \leq \p^{\lbrace 0,x\rbrace}_{\Lambda^{\g}_N,\beta}[0\nleftrightarrow \mathfrak{g}].
\end{equation*} 
To any percolation configuration $\omega$ induced by a current $\n$ on $\Lambda_{N}^{\mathfrak{g}}$ with sources $\lbrace 0,x\rbrace$, we can associate a new percolation configuration $(\hat{\omega}_e)_{e\in E_{\Lambda_{N+1}}}$ as follows:
\begin{center}
$\hat{\omega}_e =\left\{
    \begin{array}{ll}
       \omega_e & \mbox{if}~ e\in E_{\Lambda_N}\ \\
        \omega_{x\mathfrak{g}} & \mbox{if}~e=\lbrace x,y\rbrace,~x\in \partial \Lambda_N,~y\in \partial \Lambda_{N+1},~\vert x-y\vert_1 =1\ \\
        0 &\mbox{otherwise. }
    \end{array}
\right.$
\end{center}
We therefore have a surjective mapping $F:\omega\mapsto \hat{\omega}$ from the set of currents on $\Lambda_{N}^{\mathfrak{g}}$ having sources $\lbrace 0,x\rbrace$ to the set of percolation configurations on $E_{\Lambda_{N+1}}$.  
The law of $\hat{\omega}$ previously defined is therefore the push-forward measure of $ \mathbb{P}_{\Lambda_{N}^{\g}, \beta}^{\lbrace 0,x\rbrace}$ by $F$. 
Said differently, $\hat{\omega}\sim \mathbf{P}_{\Lambda_{N+1}}$, where $\mathbf{P}_{\Lambda_{N+1}}$ is the probability measure defined by $$\mathbf{P}_{\Lambda_{N+1}}(A) = \mathbb{P}_{\Lambda_{N}^{\g}, \beta}^{\lbrace 0,x\rbrace}(F^{-1}(A))$$ for any $A\in {\lbrace 0,1\rbrace}^{E_{\Lambda_{N+1}}}$.   Remark that $\mathbf{P}_{\Lambda_{N+1}}$ inherits the finite energy lower bound~(\ref{ineq:like_finite_energy}) from $\bbP_{\Lambda_{N}}^{\lbrace 0,x\rbrace}$. This allows us to reinterpret $0\nleftrightarrow \mathfrak{g}$ as the event that $0$ is disconnected from $\partial \Lambda_{N+1}$. Indeed,  observe that in order to have a connection from $0$ to $\partial \Lambda_{N+1}$ in $\hat{\omega}$, there must be a connection from $0$ to $\mathfrak{g}$ in $\omega$. This implies in particular that 
\begin{equation*}
\p^{\lbrace 0,x\rbrace}_{\Lambda^{\g}_N}[0\nleftrightarrow \mathfrak{g}] \leq \mathbf{P}_{\Lambda_{N+1}}[0\nleftrightarrow \partial \Lambda_{N+1},  0\leftrightarrow x].
\end{equation*} 
\begin{figure}[t]\label{fig:saturation_proof_alpha_star}	
\includegraphics[scale=0.75]{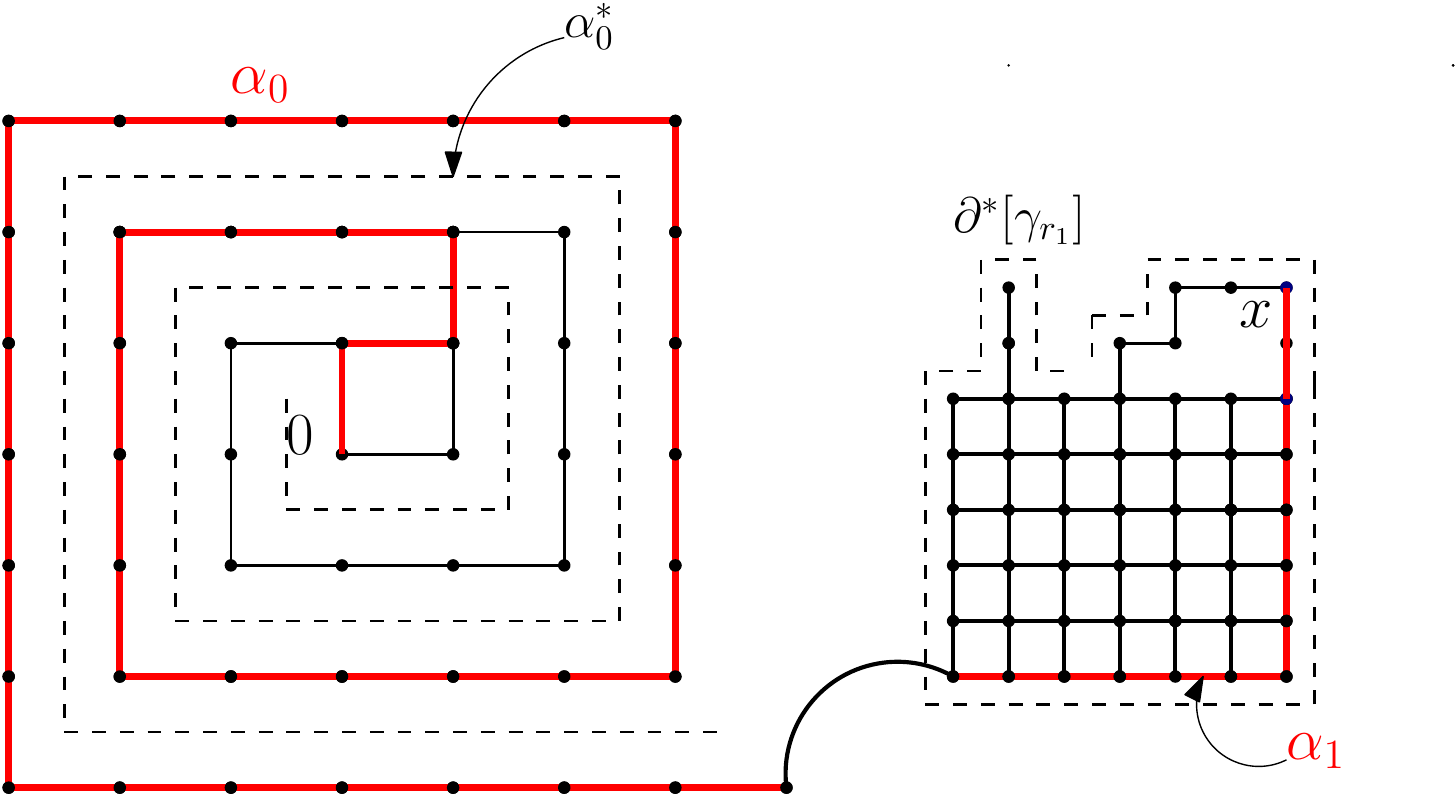}
\caption{A realization of $\mathcal{C}_{0,x}$. The open dual edges are dashed and the $\alpha_i$'s are in red.}
\end{figure}
Such an event can easily be described using dual blocking surfaces in a Peierls-like argument. 
We will call a path $\textit{basic}$ if it only uses edges of length $1$. Consider $\calC_{0,x}$ the joint cluster of $0$ and $x$. 
For any $y\in \calC_{0,x}$,  denote by $[y]$ the (random) set of points $z \in \calC_{0,x}$ such that there exists an open basic path joining $y$ to $z$.  Choose an arbitrary order on $\mathbb{Z}^d$.
Choose $\gamma=(\gamma_0,...,\gamma_n)$ joining $0$ to $x$ to be an open self-avoiding path minimal according to this order. We extract a new path from $\gamma$ using the following procedure. 
Let $r_0=0$ and 
\begin{equation*}
r_1 := \max \lbrace i ~:~\gamma_i \in [0],~0\leq i \leq n \rbrace.
\end{equation*}
For $k\geq 1$, define recursively
\begin{equation*}
r_{k+1} := \max \lbrace i ~:~ \gamma_i \in [\gamma_{r_k + 1}],~r_k<i\leq n\rbrace.
\end{equation*}
This procedure stops as soon as $r_k=n$.  Let $m=m(\gamma)$ be such that $r_m=n$. 
By construction,  for any $1\leq k< m$ we have the inclusion $\lbrace \gamma_{r_k +1},...,\gamma_{r_{k+1}}\rbrace \subset [\gamma_{r_k}]$, and the sets $\left([\gamma_{r_k}]\right)_{k\geq1}$ are all disjoint sets.  For any $k\geq 1$, there is a minimal self-avoiding basic path of open edges joining $\gamma_{r_k + 1}$ to $\gamma_{r_{k+1}}$, using only points in $[\gamma_{r_k}]$,  that is minimal with respect to the order we previously chose.  Denote by $\alpha_k$ such a path,  and by $\lambda_k$ its length.  Denote by $\alpha_0$ (respectively $\lambda_0$) the self-avoiding  basic path joining $0$ to $\gamma_{r_1}$ (respectively its length).  

We now have a new self-avoiding path joining $0$ to $x$ defined by taking the union of the paths $(\alpha_i)_{i\geq 0}$.  From now on, we will denote by $\gamma$ this new path in order to lighten the notations. To any cluster realization of the cluster $\calC_{0,x}$ one can thus associate an open path $\Gamma(\calC_{0,x})$ joining $0$ to $x$ using this procedure.
Moreover, each $\alpha_{k}$ is contained in the interior of a dual basic path of open edges. Denote by $\partial^{\ast} [\gamma_{r_k}]$ the shortest such path and by $\text{Int}(\partial^{\ast} [\gamma_{r_k}])$ its interior. We call $\partial^{\ast} [\gamma_{r_k}]$ \textit{the dual boundary of $[\gamma_{r_k}]$}.  Note that the $[\gamma_{r_k}]$'s are disjoint and each edge belonging to one of their dual boundaries can belong at most to two different boundaries.  
Since all the $[\gamma_{r_i}]$'s are connected subgraphs of a lattice and the $\alpha_i$'s are of minimal length,  there exists a family $\alpha_{1}^{\ast},...,\alpha_{m}^{\ast}$ of dual basic paths with $\vert  \alpha_i^{\ast}\vert \geq \vert \alpha_i\vert $ and $\abs{\alpha^{\ast}_i}\neq 0$ for all $0\leq i \leq n$,  such that, for every $i\in\lbrace 1,\dots,n\rbrace$, one has 
\begin{equation*}
\alpha_{i}^{\ast}\subset \left(\partial^{\ast} [\gamma_{r_i}]\cup \text{Int}(\partial^{\ast} [\gamma_{r_i}])\right)\setminus \bigcup_{j\neq i}\text{Int}(\partial^{\ast} [\gamma_{r_j}])
\end{equation*}
and such that there exists a deterministic constant $K>0$ satisfying
\begin{equation}\label{densite_positive}
\frac{\abs{\lbrace e^{\ast}\in\alpha_{i}^{\ast}: \omega_{e^{\ast}}=1\rbrace}}{\vert  \alpha_{i}^{\ast}\vert}\geq K.
\end{equation}
Notice that it is possible that $\alpha_{i}^{*}=\partial^{\ast} [\gamma_{r_i}]$. We are going to prove that there exists $C,c>0$ such that
\begin{equation*}
\mathbf{P}_{\Lambda_{N+1}} \left[0\nleftrightarrow \partial \Lambda_{N+1},  0\leftrightarrow x,  \Gamma(\calC_{0,x})= \gamma\right] \leq  \prod_{i=1}^{\vert\gamma\vert}C\sfe^{-c \beta}\beta J_{\gamma_{i-1}\gamma_{i}}.
\end{equation*}
In order to prove this inequality, we are going to use the fact that all edges in $\gamma$ are open (which will give the contribution in $\beta J_{\gamma_{i-1},\gamma_{i}})$, that all (dual) edges in $\partial^{\ast} [\gamma_{r_i}]$ are open and that there exists a strictly positive proportion of (dual) edges in $\alpha_{i}^{*}$ that are open. Fix now some $\alpha_{k}$. We are going to separate between two cases.
Firstly, assume $\abs{\partial^{\ast}[\gamma_{r_k}]} \geq \abs{\alpha_k}$. In this case, using~\eqref{ineq:like_finite_energy}, one has
\begin{equation}\label{ineq:estimate_dual_gamma}
\mathbf{P}_{\Lambda_{N+1}}(
\omega_{f^{\ast}} = 1\hphantom{,}\forall f^{\ast}\in\partial^{\ast}[\gamma_{r_k}])\leq C 2^{\abs{\partial^{\ast}[\gamma_{r_k}]}} \sfe^{-\frac{\beta}{2}  \abs{\partial^{\ast}[\gamma_{r_k}]}}\leq C\sfe^{-c\beta\abs{\partial^{\ast}[\gamma_{r_k}]}},
\end{equation}
In the first inequality,  the $\frac{1}{2}$ factor ensures that edges belonging to two different boundaries are not counted twice in the upcoming bounds. Therefore, the existence of an open dual basic path $\partial^{\ast}[\gamma_{r_k}]$ of length at least $\abs{\alpha_{k}^{\ast}}$ surrounding $\alpha_k$ is an event of probability $C\sfe^{-c\beta\abs{\alpha_{k}}}$. 

Secondly, assume that $\abs{\partial^{\ast}[\gamma_{r_k}]} < \abs{\alpha_k}$. In this case, using~\eqref{densite_positive} and\eqref{ineq:like_finite_energy}, the existence of $\alpha_k^{\ast}$ is an event with probability bounded by
\begin{equation*}
\vert \partial^{\ast} [\gamma_{r_k}]\cup \text{Int}(\partial^{\ast} [\gamma_{r_k}])\vert C\sfe^{-c\beta K \vert \alpha_{k}^{\ast} \vert}\leq C\sfe^{-c'\beta K\abs{\alpha_{k}^{\ast}}},
\end{equation*}
where we used that the number of ways of choosing open edges in $\alpha_{k}^{\ast}$ is given by $\sum_{r\geq K\vert \alpha_i^{\ast}\vert}\binom{\vert \alpha_i^{\ast}\vert}{r}$.

Putting all of this together, denoting by $A_N$ the event $\lbrace 0\leftrightarrow x\rbrace \cap \lbrace 0 \nleftrightarrow \Lambda_{N+1}\rbrace$, the union bound gives
\begin{align}
\mathbf{P}_{\Lambda_{N+1}} \left(A_N,~
  \Gamma(\calC_{0,x})= \gamma\right) &\leq \sum_{\alpha^{\ast}_{1},... ,\alpha^{\ast}_{m(\gamma)}} \prod_{k=1}^{m(\gamma)}C\sfe^{-c\beta K \left\vert  \alpha^{\ast}_{k}\right\vert}\prod_{i=1}^{\vert\gamma\vert}\beta J_{\gamma_{i-1}\gamma_{i}}\\
&\leq \left(\prod_{k=1}^{m(\gamma)} \sum_{\alpha^{\ast}_{k}} C\sfe^{-c\beta K \left\vert  \alpha^{\ast}_{k}\right\vert} \right) \prod_{i=1}^{\vert\gamma\vert}\beta J_{\gamma_{i-1}\gamma_{i}}\\
&= \prod_{k=1}^{m(\gamma)}C\sfe^{-c'\beta \vert \lambda_k \vert}  \prod_{i=1}^{\vert\gamma\vert}\beta J_{\gamma_{i-1}\gamma_{i}},
\end{align}
where, in the last line, we used that the number of paths $\alpha_{k}^{\ast}$ of length $l$ is bounded by $(2d)^{l}$ and took $\beta$ large enough.  Since $\sum_{k}\lambda_{k}\geq \frac{\abs{\gamma}}{2}$,  there exist two positive constants $C$ and $c$ such that 
\begin{equation}
\mathbf{P}_{\Lambda_{N+1}} \left[A_n,  \Gamma(\calC_{0,x})= \gamma\right] \leq  \prod_{i=1}^{\vert\gamma\vert}C\sfe^{-c \beta}\beta J_{\gamma_{i-1}\gamma_{i}}.
\end{equation}
Therefore, for any $\beta$ big enough, one has 
$$\mathbf{P}_{\Lambda_{N+1}} \left[A_n \right] \leq \sum_{n\geq0}\sum_{\gamma \in \text{SAW}_n(0,x)}\prod_{i=1}^{\vert\gamma\vert}C\sfe^{-c \beta}\beta J_{\gamma_{i-1}\gamma_{i}}\leq G_{\lambda(\beta)}^{\text{KRW}}(0,x),$$
where $\lim_{\beta\rightarrow\infty}\lambda(\beta)=0$. Since we have
\begin{equation*}
\trn_{\Lambda_{N};\beta} = \langle \sigma_0 \sigma_x \rangle_{\Lambda_{N};\beta}\p^{\emptyset,\lbrace 0,x\rbrace}_{\Lambda^{\g}_{N}, \beta}\left[0\nleftrightarrow\mathfrak{g} \right]\leq \langle \sigma_0 \sigma_x \rangle_{\Lambda_{N};\beta}\mathbf{P}_{\Lambda_{N+1}} \left[A_n \right],
\end{equation*} 
there exists a constant $c_{\beta}>0$ such that $\trn_{\Lambda_{N}; \beta} \leq c_{\beta} G_{\lambda(\beta)}^{\text{KRW}} (0,x)$ for $\beta$ big enough. Taking the limit as $N\rightarrow\infty$, one finally gets
$$\trn_{\beta} \leq c_{\beta} G_{\lambda(\beta)}^{\text{KRW}} (0,x),$$
which is the desired result.
\end{proof}
\begin{remark}
In the case of the Ising model with striclty positive magnetic field $h$, one could prove that there exists a non-trivial saturated regime in a straightforward way. Indeed, one can derive a random-current representation of the truncated two-point function in such a way that $J_{x,\mathfrak{g}}=h>0$ for any $x\in\bbZ^{d}$ and that 
\begin{equation*}
\langle\sigma_{0};\sigma_{x}\rangle_{\Lambda;\beta,h}=\langle\sigma_{0},\sigma_{x}\rangle_{\Lambda;\beta,h}\p^{\emptyset, \lbrace 0,x\rbrace}_{\Lambda^{\g};\beta,h}\left[0\nleftrightarrow\mathfrak{g} \right].
\end{equation*}
In particular, in this case, for any connection $\gamma$ from $0$ to $x$ in the right-hand side, $\gamma$ has to be disconnected from $\mathfrak{g}$ which is an event of probability of order $\sfe^{-ch\beta}$. Therefore, Lemma~\ref{lemma:saturation_comparaison_KRW} holds in this case as well, from which the desired conclusion follows.
\end{remark}
\section*{Acknowledgments}
YA is supported by the Swiss NSF grant 200021\_200422 is a member of the NCCR SwissMAP. KK thanks the Excellence Fellowship program at the University of Geneva for supporting him during his studies. Both authors very kindly thank Yvan Velenik and Sébastien Ott for useful discussions. We also thank Yvan Velenik for reading the first version of the present article and several helpful comments.

\bibliography{BIGbib}
\bibliographystyle{plain}

\end{document}